\pdfoutput=1
\documentclass[12pt]{article}
\usepackage[margin=1in]{geometry}
\usepackage{amsmath}
\usepackage{amsfonts}
\usepackage{amssymb}
\usepackage{amsthm}
\usepackage{esint}
\usepackage{mathtools}
\usepackage{array,makecell,longtable}
\usepackage{titling}
\usepackage{authblk}
\usepackage{commath}
\usepackage{xcolor}
\usepackage{graphicx}
\usepackage[toc,title,page]{appendix}
\usepackage{subcaption}
\usepackage{bbm}
\usepackage{hyperref,cite}
\usepackage{color,soul}
\usepackage{bm}
\usepackage{mathrsfs}

\usepackage{tikz}
\usepackage{enumitem}
\newtheorem{remark}{Remark}[section]
\newtheorem{theorem}[remark]{Theorem}
\newtheorem{corollary}[remark]{Corollary}
\newtheorem{lemma}[remark]{Lemma}

\newtheorem{definition}[remark]{Definition}

\numberwithin{equation}{section}
\allowdisplaybreaks

\newcommand{\p}{\partial} 

\title{Decoding a mean field game by the Cauchy data around its unknown stationary states}
\author[1,*]{Hongyu Liu}
\author[1,$\dagger$]{Catharine W. K. Lo}
\author[2,$\natural$]{Shen Zhang}
\affil[1]{Department of Mathematics, City University of Hong Kong}
\affil[2]{Department of Mathematics, Michigan State University}
\affil[*]{hongyu.liuip@gmail.com, hongyliu@cityu.edu.hk}
\affil[$\dagger$]{wingkclo@cityu.edu.hk}
\affil[$\natural$]{zhan2387@msu.edu}
\date{}

\begin{document}
\maketitle

\begin{abstract}
    In recent years, mean field games (MFGs) have garnered considerable attention and emerged as a dynamic and actively researched field across various domains, including economics, social sciences, finance, and transportation. The inverse design and decoding of MFGs offer valuable means to extract information from observed data and gain insights into the intricate underlying dynamics and strategies of these complex physical systems. This paper presents a novel approach to the study of inverse problems in MFGs by analyzing the Cauchy data around their unknown stationary states. This study distinguishes itself from existing inverse problem investigations in three key significant aspects: Firstly, we consider MFG problems in a highly general form. Secondly, we address the technical challenge of the probability measure constraint by utilizing Cauchy data in our inverse problem study. Thirdly, we enhance existing high order linearization methods by introducing a novel approach that involves conducting linearization around non-trivial stationary states of the MFG system, which are not a-priori known. These contributions provide new insights and offer promising avenues for studying inverse problems for MFGs. By unraveling the hidden structure of MFGs, researchers and practitioners can make informed decisions, optimize system performance, and address real-world challenges more effectively.

    \medskip
		
		

\end{abstract}

\section{Introduction}

\subsection{Mathematical Setup}

Mean field games (MFGs) is a powerful mathematical framework used to analyze the behavior of large populations of interacting agents, averaged-out over the number of agents. It is a powerful tool that provides valuable insights into the analysis of complex systems involving a large number of rational decision-makers, such as crowds \cite{Achdou2021econs,MFGCrowd,MFGCrowd+Econs}, financial markets \cite{MFGCrowd+Econs,CarmonaDelarue2018_1,Lacker2019finance}, traffic flows \cite{MFGCar2,MFGAutoCar2}, or social networks \cite{MFGSocialNetwork}. MFGs have gained significant attention and have become a thriving field of research due to their ability to capture emergent phenomena arising from the interactions and strategic decision-making of a vast number of individuals.

The origins of MFG theory can be traced back to the pioneering works of Caines, Huang, and Malham\'e \cite{huang2006large} and Lasry and Lions \cite{LasryLions1, LasryLions2}, who independently developed this framework to address the challenges of modeling and analyzing such systems. Since then, MFGs have proven to be a highly effective tool in understanding equilibrium states, optimal strategies, and the emergence of global patterns in various domains. As such, MFGs have gained significant attention and have become a vibrant field of research.

At the heart of MFGs is the concept of considering the averaged-out behaviors of agents as the system approaches the limit of an infinitely large population. By doing so, MFGs provide a macroscopic perspective that captures the collective dynamics and strategic interactions of the agents, while abstracting away from individual idiosyncrasies. By focusing on averaged characteristics, this macroscopic prespective of MFGs allows for a simplified yet insightful analysis of complex systems, enabling the study of equilibrium states, the interplay of optimal strategies, and the emergence of global patterns, in a computationally tractable manner.

In a game, each individual player makes decisions based on their own optimization problem while taking into account the decisions made by other players. A key feature of MFGs is the presence of an adversarial regime, where the agents' decisions are influenced by the actions of others. In this regime, a Nash equilibrium exists, representing a state in which no agent can unilaterally improve their own outcome by deviating from their chosen strategy, and is unique within the so-called monotone regime. 

This mean field equilibrium is typically illustrated by the following MFG system:
   \begin{equation}\label{eq:MFG0}
    \begin{cases}
        -\partial_t v(x,t) -\sigma\Delta v(x,t) + \mathcal{H}(x,t,\nabla v,m) = F(x,t,v,m) &\quad \text{in }Q:=\overline{\Omega}\times[0,T],\\
        \partial_t m(x,t) -\Delta(\sigma m(x,t)) - \nabla\cdot(m\nabla_p \mathcal{H}(x,t,\nabla v,m)) = 0  &\quad \text{in }Q\\
        v(x,T)=G(x,m(x,T)),\quad m(x,0)=f(x)&\quad \text{in }\Omega.
    \end{cases}
\end{equation} 
Here, $\Omega$ is a bounded Lipschitz domain in the Euclidean space $\mathbb{R}^n$ with $n\in\mathbb{N}$, $\sigma(x,t)>0$ is a positive real function describing the diffusion process, while $\mathcal{H}$ is a nonlinear Hamiltonian with real values. This system captures the dynamics of the agents' behaviors and the evolution of the population density.

Due to the significant size of the population, it is reasonable to consider the number of players approaching infinity, so we can consider just a representative player. The value function of this representative player is given by $v(x,t)$ in \eqref{eq:MFG0}, characterized by a state variable $x$ representing their relevant characteristics in the problem at hand. These characteristics could include physical location, social status, the quantity of a particular resource, and so on. This representative player then interacts with other players through its dependence on the empirical density of agents, denoted as $m$, in the sense that this interaction is solely based on the positions of the players in the state space, without considering their individual identities. Furthermore, as the number of players becomes very large, it is reasonable to assume that $m$ becomes deterministic and is unaffected by the behavior of any individual player. Then, $m(\cdot,t)\in\mathcal{P}(\Omega)$, where $\mathcal{P}(\Omega)$ is the set of Borel probability measures on $\Omega$. In this system, $F$ is the running cost function which signifies the interaction between the agents and the population; $f$ represents the initial population distribution and $G$ signifies the terminal cost. All the functions involved are real valued.

By solving this MFG system \eqref{eq:MFG0}, researchers can characterize the equilibrium state and analyze the strategic interactions among the agents. We shall delve into the well-posedness of the MFG system in Section \ref{sec:wellpose}. At the same time, while the primary application of MFG is to understand large systems, it also offers valuable tools for studying practical problems from an inverse perspective, such as inverse design and decoding MFGs. Henceforth, we aim to determine the intrinsic elements of the system \eqref{eq:MFG0}, specifically the running cost $F$ and terminal cost $G$. To achieve this, we use Cauchy data on $\Sigma:=\partial\Omega$.  Let $$\mathcal{C}_{F,G}=\left.(v(x,t),m(x,t)\text{ and their related functions})\right|_{\Sigma\times[0,T]},$$ then the measurement map is formally given by
\begin{equation}\label{eq:MeasureFormalTime}
     \mathcal{C}_{F,G} \to F,G\\\quad\text{ for all admissible solutions }v,m\text{ to \eqref{eq:MFG0}}.
\end{equation}

 Our primary focus is on the inverse design and decoding problem of recovering $F$ and $G$ in \eqref{eq:MFGStat0}, with a particular emphasis on the unique identifiability issue:
\begin{equation}
    \mathcal{C}_{F_1,G_1}=\mathcal{C}_{F_2,G_2}\quad\text{ if and only if }\quad (F_1,G_1)\equiv(F_2,G_2),
\end{equation}
where $(F_j ,G_j)$, $j = 1, 2$, are two sets of configurations.
We do not fully solve this problem in this paper but handle it in several special cases, using two technically different methods.

The scenario in which both systems (with coefficients $(F_j,G_j)$) permit the identical solution $(v_0, m_0)$ is the first one we examine.  In other words, we assume that the system admits a stable stationary solution $(v_0, m_0)$. The exact definition of stable states will be given later in Definition \ref{stable}. Since $\Sigma=\partial\Omega$ is accessible, we can have knowledge about $(v_0, m_0)$ on $\partial\Omega$, but $v_0$ and $m_0$ are unknown in the interior of $\Omega$ since the interior of $\Omega$ is inaccessible. Then for the inverse problem, we perturb these boundary values, in the form \[v_0|_{\Sigma\times[0,T]}+\epsilon \tilde{v}, \quad m_0|_{\Sigma\times[0,T]}+\epsilon \tilde{m},\] for a small enough positive constant $\epsilon$ and $\tilde{u},\tilde{m}\in C^{2+\alpha,1+\frac{\alpha}{2}}(Q)$. Corresponding to these perturbed boundary data, the MFG system has a solution $(v_{0,\epsilon},m_{0,\epsilon})$ in $\bar{\Omega}\times[0,T]$, which remains unknown in the interior of $\Omega$, but known on $\Sigma$ that actually constitutes the measured Cauchy data: 
\[\mathcal{C}_{F,G}^{0,\epsilon}=\left.(v_{0,\epsilon}(x,t),m_{0,\epsilon}(x,t)\text{ and their related functions})\right|_{\Sigma\times[0,T]}.\]

In the practical setup, we consider to decode an MFG system when the boundary of its state domain is accessible, but the interior of the state domain is inaccessible. A much feasible strategy is to wait for the MFG system to reach its stationary state, namely $(v_0, m_0)$, and then to exert the boundary perturbations, namely $v_0|_{\Sigma\times[0,T]}+\epsilon \tilde{v}$, $m_0|_{\Sigma\times[0,T]}+\epsilon \tilde{m}$, to generate the dynamically perturbed Cauchy data on the boundary for the further process of decoding and inversion. 

The measurement map is then given by 
\[\mathcal{C}_{F,G}^{0,\epsilon}\to F,G.\]
In this case, the Cauchy data, along with information regarding the first derivative of $F$ (with respect to $m$), can be used to recover $(F,G)$, as outlined in Theorem \ref{assumeF_1}. 
Here, the solutions $(v_0,m_0)$ are fixed but unknown, and we can only measure the boundary data of their perturbed states $(v_{0,\epsilon}(x,t),m_{0,\epsilon}(x,t))$. We utilize this steady state data to determine the dynamic value functions $F$ and $G$.

It is important to note that due to the nonlinearity of the system \eqref{eq:MFG0}, the system may not have a unique solution. Therefore, the second scenario we examine is when both systems admit a time-independent solution (which may be different). Then, we can show the unique identifiability solely based on the knowledge of the Cauchy data. Additionally, we recover this time-independent solution from the Cauchy data. A crucial element in achieving this result involves examining the corresponding stationary MFG system:
\begin{equation}\label{eq:MFGStat0}
    \begin{cases}
         -\sigma\Delta v(x) + \mathcal{H}(x,\nabla v,m) + \lambda = F(x,v,m) &\quad \text{in }\Omega,\\
        -\Delta(\sigma m(x)) - \nabla\cdot(m\nabla_p \mathcal{H}(x,\nabla v,m)) = 0  &\quad \text{in }\Omega,
     \end{cases}
 \end{equation}
where $\lambda$ is a constant that can be determined through the normalizations of $v$ and $m$.
Here, $\mathcal{H}$ and $F$ are the stationary limits of the corresponding functions in \eqref{eq:MFG0}, and we denote them by the same functions for simplicity.

We observe that rather than being purely stationary, solutions of the stationary MFG system \eqref{eq:MFGStat0} can in fact be considered to be relatively stationary, in the sense that while $m$ is taken to satisfy $\partial_t m=0$, the time derivative of $u$ satisfies $\partial_t u=\lambda$ for some fixed $\lambda\in\mathbb{R}$. Such relatively stable solutions are obtained when the system reaches a relatively stable equilibrium where the behaviour of the agents does not change over time. Stationary solutions have been shown to be the limit of time-dependent solutions as the time horizon, $T$, goes to infinity \cite{MFG-stationary2,MFG-stationary}. Delving into the analysis of the stationary MFG system can thus provide valuable insights into the behavior of the time-dependent system. 

As such, we consider the measurement map of the boundary data of the solution to the stationary MFG system \eqref{eq:MFGStat0}. By combining these two technically different methods, we achieve a complete recovery of both the model as well as the the stationary states. 
Further details regarding these measurement maps and results will be provided in subsequent sections.

The major novelty that distinguishes our inverse problem study from most of the existing ones lies in the following three key aspects. Firstly, we consider MFG problems in a highly general form \eqref{eq:MFG0} and derive results for this form, allowing for a comprehensive analysis and obtaining results applicable to this broad class of MFGs. This generality enhances the practical relevance of our findings and widens the scope of their application. Secondly, we propose a methodology that relaxes the probability measure constraint typically associated with inverse problems in MFGs. By utilizing Cauchy data, which encompasses all possible solutions satisfying the equations \eqref{eq:MFG0} and \eqref{eq:MFGStat0}, we naturally satisfy any constraint imposed by the system. This approach increases the flexibility of the decoding process and enables consideration of a broader range of scenarios. Thirdly, we develop a novel technique by conducting linearization around the stationary states of the MFG system. We assume that these stationary states are fixed but are not known a priori. In the stationary case, we fully derive these stationary states in addition to recovering the running cost $F$, while in the time-dependent case, we recover $F$ as well as the terminal cost $G$ independent of these unknown fixed stationary states. This departure from traditional approaches allows us to address the decoding problem in situations where the exact stationary states are unknown. This technical development opens up new avenues for decoding MFGs and provides a valuable tool for practical applications. 
Further discussion on these aspects will be presented in Section \ref{sect:technical}.

\subsection{Technical Developments and Discussion }\label{sect:technical}
MFGs is a powerful tool for analyzing complex systems involving a multitude of rational decision-makers, allowing for a macroscopic understanding of emergent phenomena and strategic interactions. Since its introduction in \cite{huang2006large,LasryLions1, LasryLions2}, MFGs has served as a strongly versatile mathematical framework capable of capturing the dynamics and strategic decision-making of a vast number of individuals in various fields, including economics, social sciences, finance and transportation, and have gained substantial attention and have become a vibrant field of research.

The well-posedness of the MFG system has been comprehensively studied in various contexts and remains an active area of active research and investigation. For more details, we refer readers to \cite{ABC2017StationaryMFG,Ambrose2022existence,cardaliaguet2010notes,cardaliaguet2015weak,cardaliaguet2019master,cirant2020short,FerreiraGomesTada2019StationaryMFGDirichlet,ferreira2021existence} and the references therein. Such a problem is commonly referred to as the forward problem.

In general, we focus on the case where the function $\mathcal{H}(x,t,p,m)$ is convex with respect to the third variable $p$. This condition implies that the first equation in \eqref{eq:MFG0} corresponds to an optimal control problem. The value function $v$ associated with a typical player satisfies the Hamilton-Jacobi-Bellman (HBJ) equation, while the population density, represented by $m$, obeys the Fokker-Planck-Kolmogorov (FPK) equation, i.e. the first and second equations in \eqref{eq:MFG0}, respectively. These two equations are coupled nonlinearly in a forward-backward manner, with the second equation in $m$ evolving forward in time and the first equation in $v$ evolving backward in time.

Given the wide range of physical applications, it is both interesting and physically meaningful to study these practical problems from the inverse perspective, including the inverse design and decoding of MFGs. Decoding MFGs refers to the process of recovering the underlying dynamics and parameters of a mean-field game system from observed data. This plays a crucial role in various applications, including economics, social sciences, and engineering, where understanding the hidden structure of MFGs can provide insights into collective behaviors, optimal strategies, and system dynamics. In this work, we focus on the determination of interacting factors in the MFG system \eqref{eq:MFG0}, such as the running cost $F$ and the terminal cost $G$. These problems are known as the inverse problems for MFGs. Despite their physical significance and broad practical uses, such problems are far less studied in the literature. 
In \cite{LiuMouZhang2022InversePbMeanFieldGames,LiuZhang2022-InversePbMFG,LiuZhangMFG3,ding2023determining,LiuZhangMFG4}, MFG inverse problems were proposed and investigated with several novel unique identifiability results established, and in \cite{klibanov2023lipschitz,klibanov2023mean1,klibanov2023mean2,klibanov2023holder,liu2023stability,imanuvilov2023lipschitz1,imanuvilov2023unique,klibanov2023coefficient1,klibanov2023coefficient2,imanuvilov2023global} both uniqueness and stability results were derived in a variety of settings that required not only boundary measurements but also internal initial- and final-time measurements. In addition, we would also like to note that there are some numerical results \cite{chow2022numerical,ding2022mean,klibanov2023convexification} as well as some recent development \cite{ren2023unique,ren2024reconstructing}. Notably, most of these works in \cite{ding2023determining,LiuMouZhang2022InversePbMeanFieldGames,LiuZhang2022-InversePbMFG,LiuZhangMFG3,ren2023unique,ren2024reconstructing,LiuZhangMFG4} employ the construction of complex geometric optics (CGO) solutions as ``probing modes" to establish uniqueness results. However, the application of this method becomes significantly more challenging when considering the probability density constraint on $m$ and the Neumann boundary conditions. Several works have attempted to address this challenge, including \cite{LiuZhang2022-InversePbMFG,LiuZhangMFG3,LiuZhangMFG4,ren2024reconstructing}. In this work, we present an alternative method for treating these constraints, by bypassing them and not imposing any boundary conditions on the system \eqref{eq:MFG0}, but instead consider the set of all possible solutions $(v,m)$ to the system. This can be understood in practical terms: Given a stock market where players can trade from various locations and through various avenues, for instance through local stock exchanges or online broker accounts, we only consider those tradings that are conducted online from a single country. In this case, it is not necessary to obey the probability density constraint and Neumann boundary conditions. The same holds if we consider the investments, sales and profits of multi-national technology companies in a single country. This can also be applied to airport management, where certain airlines are based at a particular airport, but there are also many other airlines that land or take-off from that airport. In these scenarios, the probability density constraint and Neumann boundary conditions are not required.

Indeed, the constraints associated with the MFG system \eqref{eq:MFG0}, including boundary conditions, positivity and probability density constraints, present intriguing and challenging technical novelties from a mathematical perspective. Furthermore, the MFG system \eqref{eq:MFG0} involves nonlinear equations which are coupled together in a backward-forward manner. To address the nonlinearity, a robust strategy involves employing linearization techniques. One such powerful approach is the successive/high-order linearization method, which has been extensively developed for various inverse problems related to nonlinear partial differential equations (PDEs), as demonstrated in \cite{LinLiuLiuZhang2021-InversePbSemilinearParabolic-CGOSolnsSuccessiveLinearisation} and related works. Subsequently, the successive linearization technique was further refined in \cite{LiuZhang2022-InversePbMFG,LiuZhangMFG3,LiuZhangMFG4} to incorporate the probability density constraint. Among several technical advancements, a key idea in \cite{LiuZhang2022-InversePbMFG} is to linearize around uniform distributions of the form
\begin{equation}\label{eq:sl1}
m=\frac{a}{|\Omega|}+\delta_m\quad\text{ with }\quad \int_\Omega\delta_m=0\quad\text{ for some }a\in[0,1].
\end{equation}
While this linearization technique adequately addresses the probability density constraint, only the running cost $F$ and Hamiltonian $\mathcal{H}$ are recovered in \cite{LiuZhang2022-InversePbMFG,LiuZhangMFG3,LiuZhangMFG4}, and it is required that $F$ depends locally on $m$. It is worth noting that in \cite{LiuZhang2022-InversePbMFG,LiuZhangMFG3,LiuZhangMFG4}, besides depending on $x$ and $m$, $F$ can explicitly depend on $t$ and $v$ too, i.e., the running cost can take the form $F(x, t, v, m)$. Moreover, the use of the successive linearization technique intrinsically requires that $F$ needs to be analytic, an assumption which may not necessarily hold.

In this paper, we address these issues by making full use of the structure of the MFG system, and propose a three-fold novel approach that allows for linearization around any solution, including non-trivial ones that are not known a priori. This means that the linearized systems remain coupled in nature, in contrast to the case in \cite{LiuMouZhang2022InversePbMeanFieldGames} where the linearization is conducted around the trivial solution $(0,0)$ which allows decoupling. We treat this in the stationary case, and demonstrate the complete recovery of the running cost $F$ in the analytic case, deriving the stationary solution in the process. In the time-dependent case, we perform linearization around any fixed unknown constant solution, and successfully recover the higher-order Taylor coefficients of $F$ (second-order and above), as well as fully recover the terminal cost $G$. This in particular, extends the work of \cite{LiuMouZhang2022InversePbMeanFieldGames}, which is the only other work which considers the unique determination of $G$. Finally, we provide some brief results in the case where $F$ is not analytic.

The aforementioned issues are not unique to MFG inverse problems. Similar probability density constraints arise in various other coupled or non-coupled PDE systems, as exemplified in \cite{liu2023determining,li2023inverse,LiLoCAC2024,li2024inverse}. In these works, high-order variation and high-order/successive linearization schemes have also been used to study the associated inverse boundary problems. We believe that the mathematical strategies and techniques developed in this paper offer novel perspectives on inverse boundary problems in these new and intriguing contexts, with the potential to yield theoretical and practical results of great significance.

The remainder of this paper is organized as follows. In Section \ref{sect:prelim}, we present some preliminary results, provide the main setting and state the main results of the inverse problem. Section \ref{sec:wellpose} is dedicated to the study of the forward MFG system. In Section \ref{sec:linearzation_discussion}, we develop the high-order linearization method. Sections \ref{sec:StatMFGProof} and \ref{sec:TimeMFGProof} provide the proofs of the main theorems for the stationary case and the time-dependent case, respectively.

\section{Preliminaries}\label{sect:prelim}

\subsection{A preliminary result}
Before we begin, we first give a preliminary result for the uniqueness of $F$, in the most general case.  As a model, we assume $\mathcal{H},F\in C^{2+\alpha}(Q\times\mathbb{R}\times\mathbb{R})$ for some $0<\alpha<1$ and $\sigma\in C^{\infty}(\overline{\Omega})$.

Let \begin{equation}\label{eq:measure1gen} 
\mathcal{C}^1_F=\left(\left.\left(v(x,t),m(x,t)\right)\right|_{t=\{0,T\},x\in\Omega},\left.\left(v,m,\partial_\nu v, \partial_\nu (\sigma m),\nabla_p \mathcal{H}\right)\right|_{t\in(0,T),x\in\Sigma}\right).
\end{equation}
  Consider the measurement map $\mathcal{C}^1$ given by
  $ \mathcal{C}^1_F \to F. $
Then the following result holds:
\begin{theorem}
    Suppose there exists a (weak) solution $v,m$ to the MFG system 
    \begin{equation}\label{eq:MFG1}
    \begin{cases}
        -\partial_t v(x,t) -\sigma\Delta v(x,t) + \mathcal{H}(x,t,\nabla v,m) = F(x,t,v,m) &\quad \text{in }Q:=\overline{\Omega}\times[0,T],\\
        \partial_t m(x,t) -\Delta(\sigma m(x,t)) - \nabla\cdot(m\nabla_p \mathcal{H}(x,t,\nabla v,m)) = 0  &\quad \text{in }Q.
    \end{cases}
\end{equation}
Let $\mathcal{C}^1$ be the associated measurement map \eqref{eq:measure1gen}.
If $\mathcal{C}^1$ is known, then, the integral
\[\int_Q F(x,t,v,m)m\,dx\,dt\] is known for any solution $m$ solving \eqref{eq:MFG1}.

In the case that $m=m_0=c\in \mathbb{R}^+_0$, the integral \[c\int_Q F(x,t,v,c)\,dx\,dt\] is known.
\end{theorem}

\begin{proof}

Multiplying the first equation in \eqref{eq:MFG1} of $v$ by the solution $m$ of the second equation in \eqref{eq:MFG1}, and integrating over $Q$, we have
\begin{align*}
    &\int_Q v\partial_t m -\int_\Omega v(T)m(T) + \int_\Omega v(0)m(0) \\
    &-\int_Q v\Delta(\sigma m) - \int_\Gamma m\sigma\partial_\nu v + \int_\Gamma v \partial_\nu (\sigma m)\\
    &-\int_Q v\nabla\cdot(m\nabla_p \mathcal{H}(x,t,\nabla v,m)) + \int_\Gamma m\nabla_p \mathcal{H}(x,t,\nabla v,m)\partial_\nu v \\
    &&= \int_Q F(x,t,v,m) m.
\end{align*}
But $m$ satisfies its own equation, so we obtain that
\begin{multline*}
     -\int_\Omega v(T)m(T) + \int_\Omega v(0)m(0) - \int_\Gamma m\sigma\partial_\nu v + \int_\Gamma v \partial_\nu (\sigma m) + \int_\Gamma m\nabla_p \mathcal{H}(x,t,\nabla v,m)\partial_\nu v \\= \int_Q F(x,t,v,m) m.
\end{multline*}

Observe that the terms on the left hand side are all known when $\mathcal{C}^1$ is measured. Correspondingly, the right hand side is also known.

The second result follows simply by considering the measurement map $\mathcal{C}^1$ for the constant solution $m=m_0=c$.

\end{proof}

As a corollary, we have the following result:
\begin{corollary}\label{Cor:GeneralCase}
    Suppose that $F(x,t,v,m)=\alpha m^k$ is a power of $m$ for a given $k$. Then for the measurement map $\mathcal{C}^1$ associated to $m=m_0=c\in\mathbb{R}_0^+$, $F$ can be uniquely determined.
\end{corollary}

\begin{proof}
    By the previous theorem, 
    \[\int_Q F(x,t,v,m)m\,dx\,dt=\int_Q \alpha [m(x,t)]^k\,dx\,dt=\int_Q \alpha c^k\,dx\,dt=c^k\alpha T |\Omega|\] is known. Therefore, $\alpha$ is uniquely determined, and thus $F$ is uniquely determined.
\end{proof}

\begin{remark}
    Observe that in this case, we do not impose any boundary or initial/terminal conditions on the system \eqref{eq:MFG1}, but instead measure these data.
\end{remark}

Corollary \ref{Cor:GeneralCase} strongly suggests that we will be able to obtain uniqueness results on $F$ which are more direct, if $F$ is in polynomial form and depends only on $m$. Consequently, in the remainder of this paper, we will consider $F$ and $G$ to be analytic, and generalize Corollary \ref{Cor:GeneralCase} to obtain their uniqueness results.

\subsection{Admissible class}\label{sec:admissible}
For $k\in\mathbb{N}$ and $0<\alpha<1$, we define the H\"older space $C^{k+\alpha}(\overline{\Omega})$ as the subspace of $C^{k}(\overline{\Omega})$ such that $\phi\in C^{k+\alpha}(\overline{\Omega})$ if and only if $D^l\phi$ exist and are H\"older continuous with exponent $\alpha$ for all $l=(l_1,l_2,\ldots,l_n)\in \mathbb{N}^n$ with $|l|\leq k$, where $D^l:=\partial_{x_1}^{l_1}\partial_{x_2}^{l_2}\cdots\partial_{x_n}^{l_n}$ for $x=(x_1, x_2,\ldots, x_n)$. The norm is defined as
	\begin{equation}
		\|\phi\|_{C^{k+\alpha}(\overline{\Omega}) }:=\sum_{|l|\leq k}\|D^l\phi\|_{\infty}+\sum_{|l|=k}\sup_{x\neq y}\frac{|D^l\phi(x)-D^l\phi(y)|}{|x-y|^{\alpha}}.
	\end{equation}
	If the function $\phi$ depends on both the time and space variables, we define $\phi\in C^{k+\alpha, \frac{k+\alpha}{2}}(Q)$ if $D^lD^{j}_t\phi$ exist and are H\"older continuous with exponent $\alpha$ in $x$ and $\frac{k+\alpha}{2} $ in $t$ for all  $l\in \mathbb{N}^n$, $j\in\mathbb{N}$ with $|l|+2j\leq k.$ The norm is defined as
	\begin{equation}
		\begin{aligned}
			\|\phi\|_{ C^{k+\alpha, \frac{k+\alpha}{2}}(Q)}:&=\sum_{|l|+2j\leq k}\|D^lD^j_t\phi\|_{\infty}+\sum_{|l|+2j= k}\sup_{t,x\neq y}\frac{|\phi(x,t)-\phi(y,t)|}{|x-y|^{\alpha}}\\
			&+\sum_{|l|+2j= k}\sup_{t\neq t',x} \frac{|\phi(x,t)-\phi(x,t')|}{|t-t'|^{\alpha/2}}.
		\end{aligned}
	\end{equation}
Now we introduce the admissible classes of $F$ and $G$, in a similar way as in \cite{LiuMouZhang2022InversePbMeanFieldGames}. For the completeness of this paper, we list it here.
\begin{definition}\label{AdmissClass1}
	We say that  $U(x,t,z):\mathbb{R}^n\times\mathbb{R}\times\mathbb{C}\to\mathbb{C}$ is admissible, denoted by $U\in\mathcal{A}$, if the following two conditions are satisfied:
	\begin{enumerate}[label=(\roman*)]
		\item The map $z\mapsto U(\cdot,\cdot,z)$ is holomorphic with value in $C^{2+\alpha,1+\frac{\alpha}{2}}(Q)$ for some $\alpha\in(0,1)$;
  \item $U(x,t,m_0(x,t))=0$ for all $(x,t)\in\mathbb{R}^n\times (0,T)$ for given $m_0(x,t)$.	
		\end{enumerate}
		
		Clearly, if the conditions (i) and (ii) are fulfilled, $U$ can be expanded into a power series as follows:
		\begin{equation}\label{eq:FTime}
			U(x,t,z)=\sum_{k=1}^{\infty} U^{(k)}(x,t)\frac{(z-m_0)^k}{k!},
		\end{equation}
		where $ U^{(k)}(x,t)=\frac{\p^k U}{\p z^k}(x,t,m_0)\in C^{2+\alpha,1+\frac{\alpha}{2}}(Q).$
\end{definition}

	\begin{definition}\label{AdmissClass2}
		We say that $V(x,z):\mathbb{R}^n\times\mathbb{C}\to\mathbb{C}$ is admissible, denoted by $V\in\mathcal{B}$, if it satisfies the following two conditions:
		\begin{enumerate}[label=(\roman*)]
			\item The map $z\mapsto V(\cdot,z)$ is holomorphic with value in $C^{2+\alpha}(\mathbb{R}^n)$ for some $\alpha\in(0,1)$;
			\item $V(x,m_0(x))=0$ for all $x\in\mathbb{R}^n$ and given $m_0(x)$ depending only on $x$.
		\end{enumerate}

		Clearly, if the conditions (i) and (ii) are fulfilled, $U$ can be expanded into a power series as follows:
		\begin{equation}\label{eq:G}
			V(x,z)=\sum_{k=1}^{\infty} V^{(k)}(x)\frac{(z-m_0)^k}{k!},
		\end{equation}
		where $ V^{(k)}(x)=\frac{\p^kU}{\p z^k}(x,m_0)\in C^{2+\alpha}(\mathbb{R}^n).$
	\end{definition}


The admissibility conditions in Definitions~\ref{AdmissClass1} and \ref{AdmissClass2} shall be imposed as a-priori conditions on the unknowns $F$ and $G$ in what follows for our inverse problem study.

Throughout the rest of the paper, we assume that the Hamiltonian in \eqref{eq:MFG0} is of the following form:
\begin{equation}\label{eq:ham1}
\mathcal{H}(x,p)=\frac 1 2\kappa(x,t)|p|^2, \quad p=\nabla v(x, t), 
\end{equation}
which signifies that the Lagrangian energy of the MFG system is kinetic (cf. \cite{LiuMouZhang2022InversePbMeanFieldGames}). This type of Hamiltonian widely occurs in the MFG theory (cf. \cite{ding2022mean} ).

\subsection{Main unique identifiability results}
With these definitions in hand, we are able to articulate the primary conclusions for the inverse problems, which show that
one can fully recover the running cost and terminal cost, making use of different measurement maps in different situations.

Let us first consider the stationary MFG system for 
$\sigma\equiv\kappa\equiv1$:
\begin{equation}\label{eq:zs_main}
	\begin{cases}
		-\Delta v(x)+\frac{1}{2}|\nabla v(x)|^2=F(x,m)  &  \text{ in } \Omega\\
		-\Delta (m(x))-\nabla\cdot(m(x)\nabla v(x))=0  & \text{ in } \Omega\\
		v(x)=f(x),\quad m(x)=g(x)  & \text{ in } \Sigma.
	\end{cases}
\end{equation}

We define the Cauchy set $\mathcal{C}_F^2$ of \eqref{eq:zs_main} by
\begin{equation}\label{Cauchy_set} 
	\mathcal{C}_F^2:=\{m\big|_{\Sigma},\nabla m\cdot\nu:(v,m) \text{ satisfies equation } \eqref{eq:zs_main} \},
\end{equation}
where $\nu$ is the unit outer normal of $\Omega.$

Then we are able to fully recover the running costs $F$ from Cauchy data, and the following uniqueness result holds:
\begin{theorem}\label{zs_main_result}
	Let $n\geq3$ and $F_i(x,m)\in \mathcal{B}$ $ (i=1,2)$ such that $F_i(x,m_{0,i})=0$, where $m_{0,i}$ are solutions of system \eqref{eq:zs_main} with $F(x,m)=F_i(x,m)$ respectively. Assume that $\mathcal{C}^2_{F_1}=\mathcal{C}^2_{F_2}$, then $m_{0,1}=m_{0,2}$ and $F_1=F_2.$
\end{theorem}

\begin{remark}\label{zeroCond}
    For $F\in\mathcal{B}$, $F^{(0)}(x)=0$ is a necessary condition in order to have uniqueness of $F(x,m)$ from the knowledge of $\mathcal{C}_F$ (cf. the discussion in \cite{LiuMouZhang2022InversePbMeanFieldGames}). Otherwise, we may consider the following counter-example in 1-dimension: $F_1(x,m)=x,F_2(x,m)=x^2$. Then $v_i$ is independent of $\mathcal{C}_{F_i}$, and $\mathcal{C}_{F_1}=\mathcal{C}_{F_2}$. This is a relatively well-known result in inverse boundary problems, that one cannot recover the ``source term" from the knowledge of the Cauchy data.
\end{remark}

Next, we return to the time-dependent case, given by the quadratic MFG system 
\begin{equation}\label{eq:MFG2}
    \begin{cases}
        -\partial_t v(x,t) -\sigma(x,t)\Delta v(x,t) + \frac{1}{2}\kappa(x,t)|\nabla v(x,t)|^2 = F(x,t,m) &\quad \text{in }Q,\\
        \partial_t m(x,t) -\Delta(\sigma(x,t) m(x,t)) - \nabla\cdot(\kappa(x,t)m(x,t)\nabla v(x,t)) = 0  &\quad \text{in }Q,\\
        v(x,T)=G(x,m(x,T)),\quad m(x,0)=f(x) &\quad \text{in }\Omega,
    \end{cases}
\end{equation} 
for bounded positive real functions $\sigma(x,t),\kappa(x,t)>0$, $\sigma(x)\in C^{2,2}(\bar{Q})$, $\kappa(x,t)\in C^{1,0}(\bar{Q})$.

Consider the measurement map 
$\mathcal{C}^3_{F,G}$ given by \[\mathcal{C}^3_{F,G}=\left.\left(v,\nabla v,m,\nabla m\right)\right|_{\Sigma\times(0,T)}\to F,G \]
associated to \eqref{eq:MFG2}. 
First, we define
\begin{definition}\label{stable}
Let $F\in\mathcal{A}$, $G\in\mathcal{B}$. Let $g,h\in C^{2+\alpha,1+\frac{\alpha}{2}}(\Sigma\times(0,T))$ satisfy the following  compatibility conditions:
\begin{equation}\label{compatibility conditions }
\begin{cases}
     g(x,T)=G^{(1)}(x)h(x,T)&\quad \text{ in }\Sigma\\
      -\partial_t g -\sigma \Delta g + \kappa\nabla v_0\cdot \nabla g=F^{(1)}(x,t)h &\quad \text{ in }\Sigma\\
     h(x,0)=0&\quad \text{ in }\Sigma\\
     \partial_t h -\Delta (\sigma h) - \nabla\cdot(\kappa v_0\nabla h) -\nabla \cdot (\kappa h\nabla v_0)= 0 &\quad \text{ in }\Sigma\\
\end{cases}
\end{equation}

    A solution $(v_0,m_0)$ is called a stable solution of \eqref{eq:MFG2} if for any
    $g,h$ that satisfy the compatibility conditions, the system
  \begin{equation}\label{MFG2:stable}
    \begin{cases}
        -\partial_t v^{(1)}(x,t) -\sigma \Delta v^{(1)}(x,t) + \kappa\nabla v_0\cdot \nabla v^{(1)}(x,t)= F^{(1)}(x,t)m^{(1)}(x,t) &\quad \text{in }Q,\\
        \partial_t m^{(1)}(x,t) -\Delta (\sigma m^{(1)}(x,t)) - \nabla\cdot(\kappa m_0\nabla v^{(1)}) -\nabla \cdot (\kappa m^{(1)}\nabla v_0)= 0  &\quad \text{in }Q,\\
        v^{(1)}(x,t)=g,\quad m^{(1)}(x,t)=h&\quad \text{in }\Sigma\times(0,T),\\
        v^{(1)}(x,T)=G^{(1)}(x)m^{(1)}(x,T), \quad m^{(1)}(x,0) = 0  &\quad \text{in }\Omega.
    \end{cases}
\end{equation}
admits a unique solution $(v,m)\in [ C^{2+\alpha,1+\frac{\alpha}{2}}(Q)]^2.$
\end{definition}

    Then, we are able to reconstruct the running costs and terminal costs, through passive measurements for general analytic cost functions, and the following uniqueness result holds:

\begin{theorem}\label{assumeF_1}
    For $i=1,2$, let $F_i\in\mathcal{A}$, $G_i\in\mathcal{B}$ ($i=1,2$) such that $(v_0,m_0)\in [C^{2+\alpha,1+\frac{\alpha}{2}}(Q)]^2$ is a stable solution of the following system 
    \begin{equation}\label{eq:MFG2i}
    \begin{cases}
        -\partial_t v_i(x,t) -\sigma(x,t)\Delta v_i(x,t) + \frac{1}{2}\kappa(x,t)|\nabla v_i(x,t)|^2 = F_i(x,t,m_i) &\quad \text{in }Q,\\
        \partial_t m_i(x,t) -\Delta(\sigma(x,t)m_i(x,t)) - \nabla\cdot(\kappa(x,t)m_i(x,t)\nabla v_i(x,t)) = 0  &\quad \text{in }Q,\\
        v_i(x,T)=G_i(x,m_i(x,T)),\quad m_i(x,0)=f_i(x) &\quad \text{in }\Omega,
    \end{cases}
    \end{equation}
    Let $\mathcal{C}^3_{F_i,G_i}$ be the associated passive measurements. Suppose \[\mathcal{C}^3_{F_1,G_1}=\mathcal{C}^3_{F_2,G_2},\] and $F^{(1)}_1(x,t)=F^{(1)}_2(x,t)$. Then, one has \[F_1=F_2\quad \text{ and }\quad G_1=G_2.\]
\end{theorem}

Here $(v_i,m_i)$ are two perturbations of the time-independent solution $(v_\infty,c)$. As discussed in the introduction, unlike the stationary case, we consider the same unknown time-independent solution $(v_\infty,c)$ and measure the Cauchy data of its different perturbations in the time-dependent case. Yet, this time-independent solution $(v_\infty,c)$ remains unknown, and $F$ has also not been fully recovered. However, in the case where $\sigma=\kappa=1$ and $F=F(x,m)$ is independent of time, we can apply the results of the stationary case in Theorem \ref{zs_main_result} to fully recover $F$.

\begin{corollary}
For $i=1,2$, $n\geq 3$, let $F_i\in\mathcal{A}$, $G_i\in\mathcal{B}$ ($i=1,2$) such that $(v_0,m_0)\in [C^{2+\alpha,1+\frac{\alpha}{2}}(Q)]^2$ is a stable solution of the following system:
   \begin{equation}
    \begin{cases}
        -\partial_t v(x,t) -\Delta v(x,t) + \frac{1}{2}|\nabla v(x,t)|^2 = F_i(x,m) &\quad \text{in }Q,\\
        \partial_t m(x,t) -\Delta( m(x,t)) - \nabla\cdot(m(x,t)\nabla v(x,t)) = 0  &\quad \text{in }Q,\\
        v(x,T)=G_i(x,m(x,T)),\quad m(x,0)=f_i(x) &\quad \text{in }\Omega,
    \end{cases}
\end{equation}
 Let $\mathcal{C}^3_{F_i,G_i}$ be the associated passive measurements. Suppose \[\mathcal{C}^3_{F_1,G_1}=\mathcal{C}^3_{F_2,G_2},\] 
     Then, one has \[F_1=F_2\quad \text{ and }\quad G_1=G_2.\]
\end{corollary}

\section{Well-posedness of the forward problems}\label{sec:wellpose}
Before we begin, we first investigate the well-posedness of the associated forward problems. The well-posedness of the MFG system is well-studied in recent years. For example, the author shows that the system \eqref{eq:MFG0} admits a unique classical solution with Neumann boundary condition in $\Gamma:=\Sigma\times(0,T)$ \cite{ricciardi2022master}. The regularity of the solutions to the forward problems is crucial to our inverse problem study, since we will be using the high order linearization method on the MFG system, which relies on the infinite differentiability of the system with respect to small variations around a given solution.

Recall that we assume the systems admit a stable solution. In fact, by a similar method in  \cite{LiuMouZhang2022InversePbMeanFieldGames}, we can show that if $(v_0,m_0)=(0,0)$ is a solution of \eqref{eq:MFG2}, then it is a stable solution. The case $(v_0,m_0)=(C,1)$ for some constant $C$ is discussed in \cite{LiuZhang2022-InversePbMFG}. Throughout this paper, we will always assume the system \eqref{eq:MFG2} admits a stable solution.

Then we have the following theorem:
\begin{theorem}\label{wellpose}
    Let $F\in\mathcal{A}$ and $G\in\mathcal{B}$ and suppose $(v_0,m_0)$ is a stable solution of \eqref{eq:MFG2}. Then,
    \begin{enumerate}[label=(\alph*)]	
		\item
		there exist constants $\delta>0$ and $C>0$ such that for any 
		\begin{equation*}
		\begin{aligned}
		 &g\in B_{\delta}( C^{2+\alpha,1+\frac{\alpha}{2}}(\Gamma)) :=\{f\in  C^{2+\alpha,1+\frac{\alpha}{2}}(\Gamma): \|f\|_{ C^{2+\alpha,1+\frac{\alpha}{2}}(\Gamma)}\leq\delta,(f(x,0)+m_0)|_{\Sigma}=0\},\\
          &h\in D_{\delta}( C^{2+\alpha,1+\frac{\alpha}{2}}(\Gamma)) :=\{f\in  C^{2+\alpha,1+\frac{\alpha}{2}}(\Gamma): \|f\|_{ C^{2+\alpha,1+\frac{\alpha}{2}}(\Gamma)}\leq\delta,(f(x,0)+v_0)|_{\Sigma}=0\},
		\end{aligned}    
		\end{equation*}

		the system 
  \begin{equation}\label{eq:MFG_wellpose}
    \begin{cases}
        -\partial_t v(x,t) -\sigma\Delta v(x,t) + \frac{1}{2}\kappa|\nabla v(x,t)|^2 = F(x,t,m) &\quad \text{in }Q,\\
        \partial_t m(x,t) -\Delta(\sigma m(x,t)) - \nabla\cdot(\kappa m\nabla v) = 0  &\quad \text{in }Q,\\
        v(x,t)=v_0|_{\Sigma\times(0,T)  }+h,\quad m(x,t)=m_0|_{\Sigma\times(0,T)}+g&\quad \text{in }\Sigma\times(0,T),\\
        v(x,T)=G(x,m(x,T)),\quad m(x,0)=f(x) &\quad \text{in }\Omega.
    \end{cases}
\end{equation} 

  has a solution $(v,m)\in
		[C^{2+\alpha,1+\frac{\alpha}{2}}(Q)]^2$ which satisfies
		\begin{equation}
  \begin{aligned}
      	\|(v-v_0,m-m_0)\|_{ C^{2+\alpha,1+\frac{\alpha}{2}}(Q)}:&= \|v-v_0\|_{C^{2+\alpha,1+\frac{\alpha}{2}}(Q)}+ \|m-m_0\|_{C^{2+\alpha,1+\frac{\alpha}{2}}(Q)}\\  
    &\leq C(\|g\|_{ C^{2+\alpha,1+\frac{\alpha}{2}}(\Gamma)}+\|h\|_{  C^{2+\alpha,1+\frac{\alpha}{2}}(\Gamma)}).
  \end{aligned}
		\end{equation}
		Furthermore, the solution $(v,m)$ is unique within the class
		\begin{equation}
      \{ (v,m)\in  C^{2+\alpha,1+\frac{\alpha}{2}}(Q)\times C^{2+\alpha,1+\frac{\alpha}{2}}(Q): \|(v-v_0,m-m_0)\|_{ C^{2+\alpha,1+\frac{\alpha}{2}}(Q)}\leq C\delta \}.
		\end{equation}		
		\item Define a function 
		\[
		S: B_{\delta}( C^{2+\alpha,1+\frac{\alpha}{2}}(\Gamma)\times D_{\delta}( C^{2+\alpha,1+\frac{\alpha}{2}}(\Gamma)\to C^{2+\alpha,1+\frac{\alpha}{2}}(Q)\times C^{2+\alpha,1+\frac{\alpha}{2}}(Q)
		\] by \[S(g,h):=(v,m),\]
		where $(v,m)$ is the unique solution to the MFG system \eqref{eq:MFG_wellpose}.
		Then for any $(g,h)\in B_{\delta}( C^{2+\alpha,1+\frac{\alpha}{2}}(\Gamma))^2$, $S$ is holomorphic at $(g,h)$.
	\end{enumerate}
\end{theorem}

\begin{proof}
	Let 
	\begin{align*}
        &X_0:=\{ f\in C^{2+\alpha}(\Omega): (f(x,0)+m_0(x,0))|_{\Sigma}=0 \}\\
        &X_0':=\{ f\in C^{2+\alpha}(\Omega): (f(x,T)+v_0(x,T))|_{\Sigma}=0  \}\\
        &X_1:=\{f\in  C^{2+\alpha,1+\frac{\alpha}{2}}(\Gamma ):  ((m_0+f)(x,0))|_{\Sigma  }=0\}\\
		&X_1':= \{ f\in  C^{2+\alpha,1+\frac{\alpha}{2}}(\Gamma ):  ((v_0+f)(x,T)|_{\Sigma})=0 \} , \\
        &X_2:=\{f\in  C^{2+\alpha,1+\frac{\alpha}{2}}(Q):f(x,0)_{\Sigma}=0,(-\partial_t f(x,0)-\sigma(x,0)\Delta f(x,0))|_{\Sigma}=0 \}\\
        &X_2':=\{f\in  C^{2+\alpha,1+\frac{\alpha}{2}}(Q):f(x,T)_{\Sigma}=0,(-\partial_t f(x,T)+\Delta( \sigma(x,T) f(x,T))|_{\Sigma}=0 \}\\
		&X_3:=X_0\times X_0'\times [ \{ h\in C^{\alpha,\frac{\alpha}{2}}(Q):h(x,0)|_{\Sigma}=0\}]^2,
	\end{align*} and we define a map $\mathscr{L}:X_1\times X'_1\times X_2\times X_2' \to X_3$ by that for any $(g,h,\tilde v,\tilde m)\in X_1\times X'_1\times X_2\times X_2'$,
	\begin{align*}
		&
		\mathscr{L}( g,h,\tilde v,\tilde m)(x,t)\\
		:=&\big( \tilde v(x,t)|_{\Gamma}-v_0|_{\Gamma}-g , \tilde m(x,t)|_{\Gamma}-m_0|_{\Gamma}-h , 
		-\p_t\tilde v(x,t)-\Delta \tilde v(x,t)\\ &+\frac{\kappa|\nabla \tilde v(x,t)|^2}{2}- F(x,t,\tilde m(x,t)), 
		\p_t \tilde m(x,t)-\Delta \tilde m(x,t)-\nabla\cdot(\kappa\tilde m(x,t)\nabla \tilde v(x,t))  \big) .
	\end{align*}

	First, we show that $\mathscr{L} $ is well-defined. Since the
	H\"older space is an algebra under the point-wise multiplication, we have $|\nabla v|^2, \nabla\cdot(\kappa m(x,t)\nabla v(x,t))  \in C^{\alpha,\frac{\alpha}{2}}(Q ).$
	By the Cauchy integral formula,
	\begin{equation}\label{eq:F1}
		F^{(k)}\leq \frac{k!}{R^k}\sup_{|z|=R}\|F(\cdot,\cdot,z)\|_{C^{\alpha,\frac{\alpha}{2}}(Q ) },\ \ R>0.
	\end{equation}
	Then there is $L>0$ such that for all $k\in\mathbb{N}$,
	\begin{equation}\label{eq:F2}
		\left\|\frac{F^{(k)}}{k!}m^k\right\|_{C^{\alpha,\frac{\alpha}{2}}(Q )}\leq \frac{L^k}{R^k}\|m\|^k_{C^{\alpha,\frac{\alpha}{2}}(Q)}\sup_{|z|=R}\|F(\cdot,\cdot,z)\|_{C^{\alpha,\frac{\alpha}{2}}(Q) }.
	\end{equation}
	By choosing $R\in\mathbb{R}_+$ large enough and by virtue of \eqref{eq:F1} and \eqref{eq:F2}, it can be seen that the series  converges in $C^{\alpha,\frac{\alpha}{2}}(Q )$ and therefore $F(x,t,m(x,t))\in  C^{\alpha,\frac{\alpha}{2}}(Q).$  
	
	Using the compatibility assumption and regularity conditions on $G$, we have that $\mathscr{L} $ is well-defined. Now we move to show that $\mathscr{L}$ is holomorphic. It suffices to verify that it is weakly holomorphic because  $\mathscr{L}$ is clearly locally bounded. In other words, we aim to show that the map
	$$\lambda\in\mathbb C \mapsto \mathscr{L}((g,h,\tilde v,\tilde m)+\lambda (\bar g,\bar h,\bar v,\bar m))\in X_3,\quad\text{for any $(\bar g,\bar h,\bar v,\bar m)\in X_1\times X_1'\times X_2\times X_2'$}$$
	is holomorphic. In fact, this follows from the condition that $F\in\mathcal{A}$ and $G\in\mathcal{B}$.

	Note that $ \mathscr{K}( 0,0,v_0,m_0)= 0$ because $(v_0,m_0)$ is a stable solution, so we have that $\nabla_{(\tilde v,\tilde m)} \mathscr{K} (0,0,v_0,m_0)$ is a linear isomorphism between $X_2\times X_2'$ and $X_3$. Hence, by the Implicit Function Theorem, this theorem is proved.
\end{proof}
For the stationary case, we can get the well-posedness result by the same method, one may refer to \cite{ding2023determining} for a full proof. Furthermore,  the maps of boundary data to solution are $C^{\infty}$-Fr\'{e}chet differentiable,
thus we can also derive that the corresponding Dirichlet-to-Neumann
 map is also $C^{\infty}$-Fr\'{e}chet differentiable.

\section{High order linearization}\label{sec:linearzation_discussion}
\subsection{Linearization in probability space}
In \cite{Linear_in_Tn}, the authors consider the the \textit{linearized system} of MFG system in $\mathbb{T}^n:=\mathbb{R}^n/\mathbb{Z}^n$. The this theory is further developed and it is consider in a bounded domain in \cite{linear_in_Om}. We briefly introduce the basic settings and assumptions here.

 Since we need to study the linearized system  for our inverse problem study, we define the variation of a function defined on $\mathcal{P}(\Omega) $. Recall that  $\mathcal{P}(\Omega)$ denotes the set of probability measures on $\Omega$ and let $U$ be a real function defined on $\mathcal{P}(\Omega) $.
	
	\begin{definition}\label{def_der_1}
		Let $U :\mathcal{P}(\Omega)\to\mathbb{R}$. We say that $U$ is of class $C^1$ if there exists a continuous map $K:  \mathcal{P}(\Omega)\times \Omega\to\mathbb{R}$ such that, for all $m_1,m_2\in\mathcal{P}(\Omega) $ we have
		\begin{equation}\label{derivation}
			\lim\limits_{s\to 0^+}\frac{U\big(m_1+s(m_2-m_1)-U(m_1)\big)}{s}=\int_{\Omega} K(m_1,x)d(m_2-m_1)(x).
		\end{equation}
	\end{definition}
	Note that the definition of $K$ is up to additive constants. We define  the derivative
	$\dfrac{\delta U}{\delta m}$ as the unique map $K$ satisfying $\eqref{derivation}$ and 
	\begin{equation}
		\int_{\Omega} K(m,x) dm(x)=0.
	\end{equation}
	Similarly, we can define higher order derivatives of $U$, and we refer to \cite{linear_in_Om} for more related discussion.
	Finally, we  define the Wasserstein distance between $m_1$ and $m_2$ in $\mathcal{P}(\Omega)$, which shall be needed in studying the regularity of the derivative $\dfrac{\delta U}{\delta m}$.
	\begin{definition}\label{W_distance}
		Let $m_1,m_2$ be two Borel probability measures on $\Omega$. Define
		\begin{equation}
			d_1(m_1,m_2):=\sup_{Lip(\psi)\leq 1}\int_{\Omega}\psi(x)d(m_1-m_2)(x),
		\end{equation}
		where $Lip(\psi)$ denotes the Lipschitz constant for a Lipschitz function, i.e., 
		\begin{equation}\label{eq:Lip1}
			Lip(\psi)=\sup_{x, y\in\Omega, x\neq y}\frac{|\psi(x)-\psi(y)|}{|x-y|}. 
		\end{equation}
	\end{definition}
	
		We need to point out that in Definitions~\ref{def_der_1} and \ref{W_distance}, $m$ (i.e. $m_1$ or $m_2$) is viewed as a distribution. However, in other parts of the paper, we use $m$ to denote the density of a distribution such as in the MFG system \eqref{eq:MFG0}.

Finally, the relationship between the MFG system and its linearized system can be stated as follows:
 Let $(v_1, m_1)$ and $(v_2, m_2)$ be two solutions of the Mean Field Games system, associated with the starting initial conditions $m_{0}^1$
 and $m_{0}^2$. Let $(s,\rho)$ be the solution of the linearized system  related to $(u_2, m_2)$, with initial condition $m_{0}^1-m_{0}^2$. Then we have 
the norms of $v_1-v_2-s$ and $u_1-u_2-\rho$ in suitable function spaces are bounded by 
$Cd_1(m^1_{0},m_0^2)$. For details of this theorem we also refer to \cite{linear_in_Om}. However, since we work in H\"{o}lder spaces in this paper and assume the the measure $m$  always belongs to $C^{2+\alpha}(\Omega)$ at least, this result is replaced by Theorem $\ref{wellpose}$.

Next, we shall discuss the  linearized systems in our setting and  their relationship in next subsection.

\subsection{High order linearized systems}
With the smoothness of the solutions to \eqref{eq:MFG2} at hand, we are ready to conduct the high order linearization. We will show the method for the time-dependent case, and simply state the corresponding results for the stationary case. 
Let 
\begin{equation}
\begin{aligned}
      &v(x,t)|_{\Gamma}=v_0|_{\Gamma}+\sum_{l=1}^{N}\varepsilon_lg_l|_{\Gamma}\\
      &m(x,t)|_{\Gamma}=m_0|_{\Gamma}+\sum_{l=1}^{N}\varepsilon_lh_l|_{\Gamma}
\end{aligned}
\end{equation}
where $g_l,h_l\in C^{2+\alpha,1+\frac{\alpha}{2}}(\mathbb{R}^n\times\mathbb{R})$  and $\varepsilon = (\varepsilon_1,\varepsilon_2, \dots,\varepsilon_N) \in \mathbb{R}^N$ with $|\varepsilon| = \sum_{l=1}^N |\varepsilon_l|$ small enough. Recall that $(v_0,m_0)$ is a stable solution of system \eqref{eq:MFG2} and $m_0\geq 0$, so $m(x,t)|_{\Gamma}\geq0$ for $\varepsilon$ small enough. Then, by Theorem \ref{wellpose}, there exists a unique solution $(v(x,t;\varepsilon), m(x,t;\varepsilon))$ of \eqref{eq:MFG2} and $(v(x,t;0), m(x,t;0))=(v_0,m_0)$ is the solution of \eqref{eq:MFG2} when $\varepsilon = 0$.

Let $S$ be the solution operator of \eqref{eq:MFG2} defined in Theorem \ref{wellpose}. Then there exists a bounded linear operator $A$ from $\mathcal{H}:=B_{\delta}( C^{2+\alpha,1+\frac{\alpha}{2}}(\Gamma))\times D_{\delta}( C^{2+\alpha,1+\frac{\alpha}{2}}(\Gamma))$ to $[C^{2+\alpha,1+\frac{\alpha}{2}}(Q)]^2$ such that
\begin{equation}
	\lim\limits_{\|(g,h)\|_{\mathcal{H}}\to0}\frac{\|S(g,h)-S(v_0,m_0)- A(g,h)\|_{[C^{2+\alpha,1+\frac{\alpha}{2}}(Q)]^2}}{\|(g,h)\|_{\mathcal{H}}}=0,
\end{equation} 
where $\|(g,h)\|_{\mathcal{H}}:=\|g\|_{ C^{2+\alpha,1+\frac{\alpha}{2}}(\Gamma)      }+\|h\|_{C^{2+\alpha,1+\frac{\alpha}{2}}(\Gamma)}$.
Now we consider $\varepsilon_l=0$ for $l=2,\dots,N$ and fix $f_1$.
Notice that if $F\in\mathcal{A}$ , then $F$ depends on the distribution $m$ locally. We have that 
	$$
	\dfrac{\delta F}{ \delta m}(x,m_0)(\rho(x,t)):=\left<\dfrac{\delta F}{ \delta m}(x,m_0,\cdot),\rho(x,t)\right>_{L^2}=
	F^{(1)}(x)\rho(x,t),
	 $$ 
	 up to a constant. 
Then it is easy to check that $A(g,h)\Big|_{\varepsilon_1=0}$ is the solution map of the following system which is called the first-order linearization system:
\begin{equation}\label{MFG2Linear1}
    \begin{cases}
        -\partial_t v^{(1)}(x,t) -\sigma \Delta v^{(1)}(x,t) + \kappa\nabla v_0\cdot \nabla v^{(1)}(x,t)= F^{(1)}(x,t)m^{(1)}(x,t) &\quad \text{in }Q,\\
        \partial_t m^{(1)}(x,t) -\Delta (\sigma m^{(1)}(x,t)) - \nabla\cdot(\kappa m_0\nabla v^{(1)}(x,t)) -\nabla \cdot (\kappa m^{(1)}(x,t)\nabla v_0)= 0  &\quad \text{in }Q,\\
        v^{(1)}(x,t)=g_1,\quad m^{(1)}(x,t)=h_1 &\quad \text{in }\Gamma,\\
        v^{(1)}(x,T)=G^{(1)}(x)m^{(1)}(x,T), \quad m^{(1)}(x,0) = 0  &\quad \text{in }\Omega.
    \end{cases}
\end{equation}

 In the following, we define
\begin{equation}\label{eq:ld1}
 (v^{(1)}, m^{(1)} ):=A(g,h)\Big|_{\varepsilon_1=0}. 
 \end{equation}
For notational convenience, we write
\begin{equation}\label{eq:ld2}
v^{(1)}=\partial_{\varepsilon_1}v(x,t;\varepsilon)|_{\varepsilon=0}\quad\text{and}\quad m^{(1)}=\partial_{\varepsilon_1}m(x,t;\varepsilon)|_{\varepsilon=0}.
\end{equation}
We shall utilize such notations in our subsequent discussion in order to simplify the exposition and their meaning should be clear from the context. In a similar manner, we can define, for all $l\in \mathbb{N}$, $v^{(l)} := \left. \partial_{\varepsilon_l} v \right\rvert_{\varepsilon = 0}$, $m^{(l)} := \left. \partial_{\varepsilon_l} m \right\rvert_{\varepsilon = 0}$, 
we obtain a sequence of similar systems.

For the higher orders, we consider
\[v^{(1,2)}:= \left. \partial_{\varepsilon_1} \partial_{\varepsilon_2} v \right\rvert_{\varepsilon = 0}, \quad m^{(1,2)}:= \left. \partial_{\varepsilon_1} \partial_{\varepsilon_2} m \right\rvert_{\varepsilon = 0}.\]
Similarly, $(v^{(1,2)},m^{(1,2)})$ can be viewed as the output of the second-order Fr\'echet derivative of $S$ at a specific point. By following similar calculations in deriving \eqref{MFG2Linear1}, one can show that the second-order linearization is given as follows:
\begin{equation}\label{MFGQuadraticLinear2}
    \begin{cases}
        -\partial_t v^{(1,2)} -\sigma\Delta v^{(1,2)} + \kappa\nabla v_0 \cdot \nabla v^{(1,2)} + \kappa\nabla v^{(1)} \cdot \nabla v^{(2)} \\\qquad\qquad\qquad\qquad = F^{(1)}m^{(1,2)} + F^{(2)}m^{(1)}m^{(2)} &\quad \text{in }Q,\\
        \partial_t m^{(1,2)} -\Delta (\sigma m^{(1,2)}) -\nabla\cdot(\kappa m_0\nabla v^{(1,2)}) - \nabla \cdot(\kappa m^{(2)} \nabla v^{(1)}) \\\qquad\qquad\qquad\qquad = \nabla\cdot(\kappa m^{(1)} \nabla v^{(2)}) + \nabla\cdot(\kappa m^{(1,2)} \nabla v_0)   &\quad \text{in }Q,\\
        v^{(1)}(x,T)=G^{(1)}(x)m^{(1,2)}(x,T)+G^{(2)}(x)m^{(1)}(x,T)m^{(2)}(x,T) &\quad \text{in }\Omega,\\
        m^{(1,2)}(x,0) = 0  &\quad \text{in }\Omega.
    \end{cases}
\end{equation}

Inductively, for $N\in \mathbb{N}$, we consider
\[v^{(1,2,\dots,N)}:= \left. \partial_{\varepsilon_1} \partial_{\varepsilon_2} \cdots \partial_{\varepsilon_N} v \right\rvert_{\varepsilon = 0}, \quad m^{(1,2,\dots,N)}:= \left. \partial_{\varepsilon_1} \partial_{\varepsilon_2} \cdots \partial_{\varepsilon_N} m \right\rvert_{\varepsilon = 0},\] and obtain a sequence of parabolic systems.

In the stationary case, for $m_0$ not necessarily constant, we have the following first-order linearization system, when $\sigma=\kappa\equiv1$:
\begin{equation}\label{eq:MFG2StatLinear}
	\begin{cases}
		-\Delta v^{(1)}(x)+\nabla v_0(x)\cdot\nabla v^{(1)}(x)=F^{(1)}(x)m^{(1)}(x)  & \text{ in } \Omega\\
		-\Delta (m^{(1)}(x))-\nabla\cdot(m_0(x)\nabla v^{(1)}(x))-\nabla\cdot(m^{(1)}(x)\nabla v_0(x))=0  & \text{ in } \Omega.\\
	\end{cases}
\end{equation}
The higher-order linearized systems follow similarly.

\section{Stationary MFG}\label{sec:StatMFGProof}
From Theorem \ref{assumeF_1}, we see that assuming the a priori knowledge of $F^{(1)}$ is necessary to establish unique identifiability results for the time-dependent MFG, as we will demonstrate in the next section. However, the term $F^{(1)}m$ plays an important role in the proof Theorem \ref{assumeF_1}, and it is not a necessary condition to obtain this result. In \cite{LiuMouZhang2022InversePbMeanFieldGames}, the authors show that while the constant term cannot be recovered, it is possible to recover the first-order term from the knowledge of the Dirichlet-to-Neumann map. However, their proof relies on the linearization process around the trivial solution $(v,m)=(0,0)$. In this section, we establish that the running cost $F$ can be recovered without the need for knowledge of the first-order term. The proof does not rely on the existence of a trivial solution in the stationary case, and further discussions regarding the general case will be presented in future work.

\subsection{Exponentially growing solutions of linearization system}
As we discussed in Section \ref{sec:linearzation_discussion}, we consider the first-order linearization system of \eqref{eq:zs_main} in $\Omega$:
\begin{equation}\label{eq:zs_linear}
	\begin{cases}
		-\Delta v^{(1)}+\nabla v_0\cdot\nabla v^{(1)}=F^{(1)}m^{(1)}  & \text{ in } \Omega\\
		-\Delta (m^{(1)})-\nabla\cdot(m_0\nabla v^{(1)})-\nabla\cdot(m^{(1)}\nabla v_0)=0  & \text{ in } \Omega,\\
	\end{cases}
\end{equation}
where $(v_0,m_0)$ is a solution of \eqref{eq:zs_main}. One classical result \cite{MFG-stationary,MFG-book} shows that one may choose $(v_0,m_0)$ such that
\begin{equation}\label{m=e^{-v}}
\begin{cases}
     m_0=\dfrac{e^{-v_0}}{\int_{\Omega}e^{-v_0}dx}                 &\text{ in } \Omega\\
     \partial_{\nu}m_0=\partial_{\nu} v_0=0  &\text{ in } \Sigma
\end{cases}
\end{equation}
for some constant $C>0.$ It can be shown by classical fixed point argument.
 Notice that \eqref{m=e^{-v}} implies that 
\begin{equation}\label{eq:StatRelation}
	\nabla (m_0)=-m_0\nabla v_0.
\end{equation}
Expanding the second equation in \eqref{eq:zs_linear} and substituting in the first equation in \eqref{eq:zs_linear}, we have
\begin{equation}\label{calc}
	\begin{aligned}
		0&=-\Delta m^{(1)}-(\nabla v_0\cdot\nabla m^{(1)}+ m^{(1)} \Delta v_0)-(m_0\Delta v^{(1)}+\nabla m_0\cdot\nabla v^{(1)})\\
		&=-\Delta m^{(1)}-(\nabla v_0\cdot\nabla m^{(1)}+ m^{(1)} \Delta v_0) -[m_0(\nabla v_0\cdot\nabla v^{(1)}- F^{(1)}m^{(1)}) +\nabla m_0\cdot\nabla v^{(1)}]\\
		&=-\Delta m^{(1)}-(\nabla v_0\cdot\nabla m^{(1)}+(\Delta v_0-F^{(1)}m_0) m^{(1)})-[(m_0\nabla v_0+\nabla m_0)\cdot\nabla v^{(1)}]\\
		&=-\Delta m^{(1)}-\nabla v_0\cdot\nabla m^{(1)}-\left(\Delta v_0-F^{(1)}m_0\right) m^{(1)},
	\end{aligned}
\end{equation}
where the last step makes use of \eqref{eq:StatRelation}.

From this last step, the next natural step is to focus on PDEs in the form
\begin{equation}\label{semi-linear}
	\Delta m\pm\nabla v_0\cdot\nabla m+ qm=0,
\end{equation}
where $\nabla v_0\in [C^1(\Omega)]^n$ and $q\in L^{\infty}(\Omega).$

\begin{theorem}\label{thm:cgo}
	There exist solutions of \eqref{semi-linear} in the form
	\begin{equation}\label{cgo}
		m=e^{\xi\cdot x\mp\frac{1}{2}v_0}(1+\omega(x;\xi)),
	\end{equation}
	where $\xi\in \mathbb{C}^n$, $\xi\cdot \xi=0$ is a complex vector and $\omega(x;\xi)\in H^1(\Omega)$ satisfies the following decay condition:
	\begin{equation}\label{decay}
		\|\omega(x;\xi)\|_{L^2(\Omega)}\leq C|\xi|^{-1}
	\end{equation}
	for some constant $C>0.$
\end{theorem}
The proof of this theorem is based on the following lemma, which has been proved in \cite{sun_cgo}.
\begin{lemma}\label{zs_main_lemma}
	Let $L_{\xi}:=\Delta+ 2\xi\cdot\nabla$. Then the operator $L_{\xi}$ admits a bounded
	inverse $L_{\xi}^{-1}: L^2(\Omega)\to H^1(\Omega)$. If $f\in L^{\infty}(\Omega)$ and 
	$v=L^{-1}_{\xi}(f)\in H^1(\Omega)$, then
	$$\|v\|_{L^2(\Omega)}\leq C|\xi|^{-1}\|f\|_{L^2(\Omega)},$$
	$$\|\nabla v\|_{L^2(\Omega)}\leq C\|f\|_{L^2(\Omega)},$$
	where $C$ is constant only depends on $\Omega.$
\end{lemma}
Now we give the proof of Theorem \ref{thm:cgo}. The idea of this proof is borrowed from \cite{sun_cgo}. 
\begin{proof}
	We will prove this theorem for the case
	\begin{equation}\label{semi-linear'}
		\Delta m+\nabla v_0\cdot\nabla m+ qm=0,
	\end{equation}
	and a similar argument can be employed to obtain the result for the other cases.
	
	Let $m=e^{\xi\cdot x-\frac{1}{2}v_0}(1+\omega(x;\xi))$. By direct computation,  \eqref{semi-linear'} is equivalent to
	\begin{equation}\label{for_omega}
		\Delta \omega+2\xi\cdot\nabla \omega+(q-\frac{1}{4}|\nabla v_0|^2-\frac{1}{4}\Delta v_0)(1+\omega)=0.
	\end{equation}
	Let $H:=q-\frac{1}{4}|\nabla v_0|^2-\frac{1}{4}\Delta v_0$ and $L_{\xi}:=\Delta+ 2\xi\cdot\nabla$. Then \eqref{for_omega} can be rewritten as
	\begin{equation}\label{operator_form}
		L_{\xi}\omega+H\omega=-H.
	\end{equation}
	By Lemma \ref{zs_main_lemma}, we take the action of $L^{-1}_{\xi}$ on both sides of \eqref{operator_form}, to obtain
	\begin{equation}\label{operator_form_2}
	     ( I+L^{-1}_{\xi}\circ H)\omega=-L^{-1}_{\xi}(H).
	\end{equation}
	Lemma \ref{zs_main_lemma} also implies that $L^{-1}_{\xi}\circ H$ is a bounded operator from 
	$L^2(\Omega)$ to $L^2(\Omega)$ and 
	$$\| L^{-1}_{\xi}\circ H  \|\leq C \|H\|_{L^{\infty}}|\xi|^{-1}$$
	for some constant $C$. Therefore, $(I+L^{-1}_{\xi}\circ H)^{-1} $ exists for large enough $|\xi|$. In other words, there exists a unique $L^2$-solution $\omega$ of \eqref{operator_form}, and the first part of Theorem \ref{thm:cgo} is proved.
	
	Notice that $\omega$ also satisfies 
	\begin{equation}\label{operator_form_3}
		\omega=L^{-1}_{\xi}(-H(1+\omega)).
	\end{equation}
	Since $H$ is bounded and by Lemma \ref{zs_main_lemma} again, there exists a constant $C>0$ such that
	$$\|\omega(x;\xi)\|_{L^2(\Omega)}\leq C|\xi|^{-1},$$
        $$ \|\omega(x;\xi)\|_{H^1(\Omega)}\leq C $$
	when $|\xi|$ is large enough.
\end{proof}
Since we mainly work in H\"{o}lder space and we only get exponentially growing solutions in $H^1(\Omega)$, we still need the following denseness property.
\begin{lemma}\label{denseness}
    	Let $v_0\in C^1(\Omega)$ and $q\in L^{\infty}(\Omega)$. For any solution $ m\in H^1(\Omega)$ to 
	\begin{equation}\label{Runge 1}
	\Delta m+\nabla v_0\cdot\nabla m+ qm=0,
	\end{equation}
	and any $\eta>0$, there exists a solution  $\hat{m}\in C^{2+\alpha}(\Omega)$ to \eqref{Runge 1} such that
	$$\|m-\hat{m}\|_{L^2(\Omega)}\leq \eta. $$
\end{lemma}
 \begin{proof}
     Since $m\in H^1(\Omega)$, we have $m|_{\Sigma}\in L^2(\Sigma)$. Then by using the fact that $C^{\infty}(\Sigma)$ is dense in $L^2(\Sigma)$, there exist $M\in C^{\infty}(\Sigma)$ such that $\|M- m|_{\Sigma}\|_{L^2(\Omega)}\leq \eta$ for any $\eta>0.$

     Let $\hat{m}$ be the solution of
     \begin{equation*}
         \begin{cases}
             \Delta \hat{m}+\nabla v_0\cdot\nabla \hat{m}+ q\hat{m}=0 &\text{ in } \Omega\\
             \hat{m}=M &\text{ in }\Sigma.
         \end{cases}
     \end{equation*}
    Then $\hat{m}\in C^{2+\alpha}(\Omega)$ and $m-\hat{m}$ satisfies
    \begin{equation*}
          \begin{cases}
             \Delta (m-\hat{m})+\nabla v_0\cdot\nabla (m-\hat{m})+ q(m-\hat{m})=0 &\text{ in } \Omega\\
             (m-\hat{m})=m|_{\Sigma}-M &\text{ in }\Sigma.
         \end{cases}
    \end{equation*}
    Therefore,  for any $\eta>0$,
	$$\|m-\hat{m}\|_{L^2(\Omega)}\leq \|m|_{\Sigma}-M\|_{L^2(\Sigma)}\leq\eta. $$
 \end{proof}   
\subsection{Proof of Theorem \ref{zs_main_result}}

Now we are ready to give the proof of Theorem \ref{zs_main_result}.
\begin{proof}
Let $(v_{0,i},m_{0,i})$ be the solutions of \eqref{eq:zs_main} which satisfy \eqref{m=e^{-v}} with corresponding running cost $F_i(x,m)$, for $i=1,2$.

Since $F_i(x,m_{0,i})=0$ (since $F\in\mathcal{B}$, cf. Remark \ref{zeroCond}), we have 
\begin{equation}\label{v0eq}
    \begin{cases}
        -\Delta v_{0,i}+\frac{1}{2}|\nabla v_{0,i}|^2=0 & \text{ in }\Omega\\
        \partial_{\nu}v_{0,i}=0 &\text{ in }\Sigma.
    \end{cases}
\end{equation}
This shows that $\nabla v_{0,1}=\nabla v_{0,2}:=\nabla v_0$, and then $\Delta v_{0,1}=\Delta v_{0,2}:=\Delta v_0$. By \eqref{m=e^{-v}}, we can also let $m_0:=m_{0,1}=m_{0,2}.$ By the normalization resulting from $\lambda$ in \eqref{eq:MFGStat0}, we have the uniqueness of the stable state $(v_0,m_0)$.

As discussed in Section \ref{sec:linearzation_discussion}, the first-order linearization system of \eqref{eq:zs_main} in $\Omega$ is given, for $i=1,2$, by:
\begin{equation}\label{eq:zs_linear_12}
	\begin{cases}
		-\Delta v^{(1)}_i+\nabla v_{0}\cdot\nabla v^{(1)}_i=F_i^{(1)}m^{(1)}_i  & \text{ in } \Omega\\
		-\Delta m^{(1)}_i-\nabla\cdot(m_{0,i}\nabla v^{(1)}_i)-\nabla\cdot(m^{(1)}_i\nabla v_{0})=0  & \text{ in } \Omega.\\
	\end{cases}
\end{equation}
We have shown in \eqref{calc} that \eqref{eq:zs_linear_12} implies that
\begin{equation}
	\Delta m^{(1)}_i+\nabla v_{0}\cdot\nabla m^{(1)}_i+(\Delta v_{0}-F_i^{(1)}m_{0}) m^{(1)}_i=0.
\end{equation}
Let $\overline{m}:=m^{(1)}_1-m^{(1)}_2$. Since $\mathcal{C}^2_{F_1}=\mathcal{C}^2_{F_2}$, we have
\begin{equation}\label{m1-m2}
	\begin{cases}
		\Delta \overline{m} +\nabla v_{0}\cdot \nabla \overline{m}+\overline{m}\Delta v_{0}- F_1^{(1)}m_{0}\overline{m}=(F_1^{(1)}m_{0}-F_2^{(1)}m_{0})m_2 & \text{ in } \Omega\\ 
		\overline{m}=\partial_{\nu}\overline{m}=0,& \text{ in } \Sigma. 
	\end{cases}
\end{equation}
Let $u$ be a solution of the following equation:
\begin{equation*}
		\Delta u-\nabla v_{0}\cdot\nabla u-F_1^{(1)}m_{0}u=0.
\end{equation*}
Then we have
\begin{equation}\label{integral_by_part}
	\begin{aligned}
		\int_{\Omega}(F_1^{(1)}m_{0}-F_2^{(1)}m_{0})m_2u \, dx 
         &=\int_{\Omega} (\Delta \overline{m}+\nabla v_{0}\cdot\nabla \overline{m} +\overline{m}\Delta v_{0}- F_1^{(1)}\overline{m}m_0)u\, dx\\
		&=\int_{\Omega} \overline{m}\Delta u-\overline{m}\nabla v_{0}\cdot\nabla u-\overline{m}u\Delta v_{0}+\overline{m}u\Delta v_{0}-F_1^{(1)}\overline{m}m_{0}u\, dx\\
		&\quad +
		\int_{\Sigma}u\partial_{\nu}\overline{m}-\overline{m}\partial_{\nu}u      +\overline{m}u\partial_{\nu}v_{0}\,ds\\
		&=\int_{\Omega}\overline{m}(\Delta u-\nabla v_{0}\cdot\nabla u-F_1^{(1)}m_{0}u)\,dx\\
		&=0,
	\end{aligned}
\end{equation}
since $\mathcal{C}^2_{F_1}=\mathcal{C}^2_{F_2}$.

Now by Theorem \ref{thm:cgo}, there exist $\xi_i\in \mathbb{C}^n$ with $\xi_i\cdot\xi_i=0$
 such that 
\begin{equation}
	\begin{aligned}
	&\hat{m}_2=e^{\xi_1\cdot x-\frac{1}{2}v_0}(1+\omega_1(x;\xi_1))\\
	&u=e^{\xi_2\cdot x+\frac{1}{2}v_0}(1+\omega_2(x;\xi_2)),
	\end{aligned}
\end{equation}
and $\omega_j(x;\xi_j)$, $j=1,2$, satisfy the decay conditions
$$\lim\limits_{|\xi_i|\to\infty}\|\omega_j \|_{L^2(\Omega)}=0.$$
Let $k\in\mathbb{R}^n\backslash\{0\}$, then there exists $a,b\in\mathbb{R}^n,|a|=|b|=|k|$ such that $\{k,a,b\}$ is an orthogonal basis of $\mathbb{R}^n$. Now we choose
\begin{equation}\label{choose_xi}
	\begin{aligned}
		&\xi_1=\frac{i}{2}k+\left(\sqrt{R^2+\frac{1}{16}}+i\sqrt{R^2-\frac{1}{16}}\right)a+\left(\sqrt{R^2+\frac{1}{16}}-i\sqrt{R^2-\frac{1}{16}}\right)b\\
		&\xi_2=\frac{i}{2}k-\left(\sqrt{R^2+\frac{1}{16}}+i\sqrt{R^2-\frac{1}{16}}\right)a-\left(\sqrt{R^2+\frac{1}{16}}-i\sqrt{R^2-\frac{1}{16}}\right)b,
	\end{aligned}
\end{equation}
where $R\in\mathbb{R}.$
Then we have, for $j=1,2$,
$$\xi_j\cdot\xi_j=0,$$
$$\xi_j\cdot\overline{\xi_j}=\frac{1}{4}|k|^2+4R^2|k|^2.$$
By using this construction, \eqref{integral_by_part} and Lemma \ref{denseness} implies that
\begin{equation}
	\int_{\Omega} (F_1^{(1)}m_0-F_2^{(1)}m_0)e^{ik\cdot x}(1+\omega_1(x;\xi_1))(1+\omega_2(x;\xi_2))\,dx=0.
\end{equation}
Letting $R\to\infty$, we have $$ 	\int_{\Omega} (F_1^{(1)}m_0-F_2^{(1)}m_0)e^{ik\cdot x}\,dx=0.$$
Since $k\in\mathbb{R}^n\backslash\{0\}$ is an arbitrary vector, we have
 $F_1^{(1)}(x)m_0(x)-F_2^{(1)}(x)m_0(x)=0.$ and since $m_0(x)>0$ for all $x\in\Omega$ by construction in \eqref{m=e^{-v}}, we have
 $$F_1^{(1)}(x)=F^{(1)}_2(x),$$
 and we obtain the uniqueness of the first-order Taylor coefficients of $F$. 

 For the higher-order Taylor coefficients, we apply higher-order linearization and use the same method as in the proof of Theorem \ref{assumeF_1}, in the case of $\sigma=\kappa\equiv1$, which will be detailed in the next section. Consequently, we obtain $F_1(x,m)=F_2(x,m).$
\end{proof}

\section{Time-dependent MFG}\label{sec:TimeMFGProof}
\subsection{Unique continuation principle}

Next, we move on to prove the main result for the time-dependent MFG system in Theorem \ref{assumeF_1}. Before we do so, we first prove a unique continuation principle that is necessary for our proof. This result is an easy generalization of the result given in \cite{imanuvilov2023unique}, and we only give a sketch of the proof to highlight the main differences.

\begin{theorem}\label{UCP}

Let $\Gamma'\subset \Sigma$ be an arbitrarily chosen
non-empty relatively open sub-boundary. For $\sigma\in C^2(\bar{Q})$, $h\in L^\infty(Q)$, we assume that $(v_i,m_i) \in H^{2,1}(Q)\times H^{2,1}(Q)$ 
satisfy, for $i=1,2$, 
\begin{equation}\label{eq:UCPeq}
    \begin{cases}
        -\partial_t v_i(x,t) -\mathcal{L}_1 v_i(x,t) = E_i(x,t)+h(x,t)m_i(x,t) &\quad \text{in }Q,\\
        \partial_t m_i(x,t) -\mathcal{L}_2 m_i(x,t) + \mathcal{L}_3 v_i(x,t) = G_i(x,t)  &\quad \text{in }Q
    \end{cases}
\end{equation}
for some regular second-order elliptic operators $\mathcal{L}_1$, $\mathcal{L}_2$ and $\mathcal{L}_3$, such that the coefficients (depending on $x$ and $t$) of $\mathcal{L}_3$ are bounded,
and
\[\begin{cases}
v_i, \nabla v_i, \Delta v_i \in L^{\infty}(Q),\\
m_i, \nabla m_i \in L^{\infty}(Q), \quad 
\partial_t(v_1-v_2),\partial_t(m_1-m_2) \in L^2((\Sigma\backslash\Gamma')\times(0,T)).
\end{cases}\]
Then $v_1=v_2$, $\nabla v_1 = \nabla v_2$, $m_1=m_2$ and $\nabla m_1
= \nabla m_2$ on $\Gamma' \times (0,T)$ implies $v_1=v_2$ and $m_1=m_2$ 
in $Q$.

\end{theorem}

\begin{proof}
    The proof follows similarly as in \cite{imanuvilov2023unique}. Indeed, for arbitrarily fixed $t_0\in (0,T)$ and $\delta>0$ such that 
$0 < t_0 - \delta \le t_0 + \delta < T$, consider $I = (t_0-\delta, t_0+\delta)$ and $Q_I = \Omega \times I.$ Set  
\[
P_k w(x,t):= \partial_tw + (-1)^k a(x,t)\Delta w + R(x,t,w),
\quad k=1,2,
\]
where $a \in C^{2}(\overline{Q_I})$, $>0$ on $\overline{Q_I}$, and 
\begin{equation}\label{UCPeq1}
| R(x,t,w)| \leq C_0(| w(x,t)|
+ |\nabla w(x,t)|), \quad (x,t)\in Q_I.
\end{equation}
Moreover, let
\begin{equation}\label{UCPeq2}
    \varphi(x,t) = e^{\lambda(d(x) - \beta (t-t_0)^2)},
\end{equation}
where $\lambda>0$ is a sufficiently large parameter, $\beta>0$ is arbitrarily given, and
$d \in C^2(\overline{\Omega})$ is such that 
\begin{equation}
d>0 \quad \text{in }\Omega, \quad | \nabla d| > 0 \quad 
\text{on $\overline{\Omega}$}, \quad d=0 \quad\text{on }\partial\Omega\setminus
\Gamma, \quad \nabla d\cdot \nu \le 0 \quad \text{on }\partial\Omega\setminus 
\Gamma,
\end{equation}  which is known to exist. Then the following Carleman estimate (Lemma 1 of \cite{imanuvilov2023unique}) is known to hold: 
For $k=1,2$, there exist constants $s_0>0$ and $C_1, C_2>0$ such that 
\begin{equation}\label{UCPeq4}
    \int_{Q_I} \left( \frac{1}{s}(| \partial_tw|^2
+ | \Delta w|^2) + s |\nabla w|^2 
+ s^3| w|^2 \right)e^{2s\varphi} dxdt
\le Cs^4\int_{Q_I} | P_kw|^2 e^{2s\varphi} dxdt
+ C\mathcal{B}(w),
\end{equation}
for all $s > s_0$ and $w\in H^{2,1}(Q_I)$ satisfying 
$w\in H^1(\partial\Omega\times I)$, where
\begin{multline*}
 \mathcal{B}(w) := e^{Cs}\Vert w\Vert^2_{H^1(\Gamma\times I)}
+ s^3\int_{(\partial\Omega\setminus \Gamma)\times I}
(|w|^2 + | \nabla_{x,t}w|^2) e^{2s} dSdt\\
+ s^2\int_{\Omega} (| w(x,t_0-\delta)|^2
+ | \nabla w(x,t_0-\delta)|^2
+ | w(x,t_0+\delta)|^2
+ | \nabla w(x,t_0+\delta)|^2) e^{2s\varphi(x,t_0-\delta)} dx.
\end{multline*}

Next, we expand \eqref{eq:UCPeq} in the form 
    \begin{equation}\label{UCPeq5}
    \begin{cases}
        -\partial_t v_i(x,t) -\alpha\Delta v_i(x,t) + R_1(x,t,v_i) = E_i(x,t)+h(x,t)m_i(x,t) &\quad \text{in }Q,\\
        \partial_t m_i(x,t) - \beta \Delta m_i(x,t) +R_3(x,t,m_i) = \gamma\Delta v_i(x,t) + R_2(x,t,v_i) + G_i(x,t)  &\quad \text{in }Q,
    \end{cases}
\end{equation} for some bounded coefficients $\alpha(x,t)$, $\beta(x,t)$ and $\gamma(x,t)$, so that \[|R_1(x,t,v_i)|,|R_2(x,t,v_i)|\leq C\sum_{k=0}^1|\nabla^k v_i(x,t)|, \quad |R_3(x,t,m_i)|\leq C\sum_{k=0}^1|\nabla^k m_i(x,t)|\] 
for some different constants $C$. Then, taking the difference of the two systems for $i=1,2$, we can apply the Carleman estimate \eqref{UCPeq4}, and choose $s>0$ large enough to absorb terms on the right hand side of \eqref{UCPeq4}, to obtain the following Carleman estimate: 
\begin{multline}\label{UCPeq6}
 \int_{Q_I} \biggl( | \partial_t(v_1-v_2)|^2 
+ | (\Delta (v_1-v_2)|^2 
+ s^2|\nabla (v_1-v_2)|^2 + s^4| v_1-v_2|^2 
\\+ \frac{1}{s}(| \partial_t(m_1-m_2)|^2 
+ | \Delta (m_1-m_2)|^2) 
+ s|\nabla (m_1-m_2)|^2 + s^3|m_1-m_2|^2 \biggr)e^{2s\varphi} dxdt
\\\leq C_3\int_{Q_I} (s| F-\widetilde{F}|^2 + | G-\widetilde{G}|^2)e^{2s\varphi} dxdt\\+  C_4s(\mathcal{B}(v_1-v_2) + \mathcal{B}(m_1-m_2)) \quad \text{for all $s>s_0$}
\end{multline}
for some constants $s_0>0$ and $C_3,C_4>0$.

Next, we arbitrarily choose $t_0\in (0,T)$ and $\delta>0$ such that 
$0<t_0-\delta<t_0+\delta<T$.
We define 
\begin{equation}\label{UCPeq7}
d_0:= \min_{x\in \overline{\Omega}} d(x),\quad
d_1:= \max_{x\in \overline{\Omega}} d(x),\quad
0 < r < \left( \frac{d_0}{d_1}\right)^{\frac{1}{2}} < 1.
\end{equation}
We observe that to prove Theorem \ref{UCP}, it suffices to measure $v,m,\nabla v,\nabla m$ on 
$\Gamma \times (t_0-\delta,\, t_0+\delta)$ and prove the result in $\Omega \times (t_0-r\delta,\, t_0+r\delta)$. 
Indeed, since $t_0 \in (0,T)$ and $\delta>0$ can be arbitrarily chosen and the Carleman estimate is invariant with respect to $t_0$ provided that $0<t_0-\delta<t_0+\delta<T$, we can change $t_0$ over $(\delta, T-\delta)$ to obtain
$v_1=v_2$ and $m_1=m_2$ in $\Omega \times ((1-r)\delta,\, T-(1-r)\delta)$.
Since $\delta>0$ can be arbitrary, this means that 
$v_1=v_2$ and $m_1=m_2$ in $\Omega\times (0,T)$.

Consequently, we choose $\beta > 0$ in the weight of the Carleman estimate such that 
\[
\frac{d_1-d_0}{\delta^2 - r^2\delta^2} < \beta < \frac{d_0}{r^2\delta^2}.\]
Considering only the dominating terms in \eqref{UCPeq7} for large $s>0$, and observing that 
\[\varphi(x,t) = e^{\lambda(d(x) - \beta(t-t_0)^2)}
\geq e^{\lambda(d_0-\beta r^2\delta^2)}=:\mu_1\quad \text{in }
\Omega \times (t_0-r\delta,\, t_0+r\delta),\] 
we shrink the region of integration of the left hand side of \eqref{UCPeq7} to $\Omega \times (t_0-r\delta,\, t_0+r\delta)$ to obtain that 
\begin{multline}\label{UCPeq8}\norm{v_1-v_2}^2_{L^2(\Omega \times (t_0-r\delta,\, t_0+r\delta))}
+ \norm{m_1-m_2}^2_{L^2(\Omega \times (t_0-r\delta,\, t_0+r\delta))}
\\\leq C_5s^2M_1e^{-2s(\mu_1-1)} + C_6s^2M_2e^{-2s(\mu_1-\mu_2)}\end{multline}
for all large $s>0$ for some constants $C_5,C_6>0$, where $\mu_2:= e^{\lambda(d_1-\beta\delta^2)}$ and
\begin{align*}
& M_1:= \sum_{k=0}^1 
(\Vert\nabla^k_{x,t}(v_1-v_2)\Vert^2
_{L^2((\partial\Omega\setminus \Gamma)\times I)}
+ \Vert\nabla^k_{x,t}(m_1-m_2)\Vert^2
_{L^2((\partial\Omega\setminus \Gamma)\times I)}), \\
& M_2:= \sum_{k=0}^1 (\Vert (v_1-v_2)(\cdot,t_0 + (-1)^k\delta) \Vert^2
_{H^1(\Omega)} 
+ \Vert (m_1-m_2)(\cdot,t_0 + (-1)^k\delta) \Vert^2_{H^1(\Omega)}.
\end{align*}
Observing that 
$\mu_1 > \max\{ 1, \, \mu_2\}$, we take $s \to \infty$ in \eqref{UCPeq8} to obtain 
$v_1=v_2$ and $m_1=m_2$ in $\Omega \times (t_0-r\delta,\, t_0+r\delta)$.
Thus the proof is complete.
\end{proof}

\subsection{Proof of Theorem \ref{assumeF_1}}
With the unique continuation principle in hand, we are ready to give the proof of the main result in Theorem \ref{assumeF_1}. Before we begin, we first observe that in order to make use of the method of high order linearization in our proof, we linearize around the same stationary state $(v_0,m_0)$. In the stationary case, we have shown, in the beginning of the proof of Theorem \ref{zs_main_result}, that when $\mathcal{C}^2_{F_1}=\mathcal{C}^2_{F_2}$, it must hold that $(v_{0,1},m_{0,1})=(v_{0,2},m_{0,2})$. In the time-dependent case, we are unable to show this directly, because the time-dependent Cauchy problem requires more data compared to the stationary Cauchy problem. Instead, here, we take this result as an assumption. Then this assumption can be understood as follows: Given that the MFG system has stabilized to an unknown fixed stationary state $(v_0,m_0)$, we perturb it and measure the corresponding Cauchy data $\mathcal{C}^3_{F,G}$. In this sense, we can consider the linearization about the same state $(v_0,m_0)$, so that for $i=1,2$, we can consider the corresponding linearized systems \eqref{MFG2Linear1} and \eqref{MFGQuadraticLinear2}. 

With this consensus in mind, we can continue with the proof of Theorem \ref{assumeF_1}.

\begin{proof}
    The proof is based on the strong unique continuation principle in Theorem \ref{UCP}. Indeed, let $\tilde{v}=v_1^{(1)}-v_2^{(1)}$ and $\tilde{m}=m_1^{(1)}-m_2^{(1)}$. Then, $(\tilde{v},\tilde{m})\in [H^{2,1}(Q)]^2$ is the solution to the system 
    \begin{equation}\label{eq:MFG2Diff1}
    \begin{cases}
        -\partial_t \tilde{v}(x,t) -\sigma\Delta \tilde{v}(x,t) + \kappa\nabla v_0(x,t)\cdot \nabla \tilde{v}(x,t) = F^{(1)}(x)\tilde{m}(x,t) &\quad \text{in }Q,\\
        \partial_t \tilde{m}(x,t) -\Delta(\sigma \tilde{m}(x,t)) - c\nabla\cdot(\kappa \nabla \tilde{v}(x,t)) -\nabla \cdot (\kappa\tilde{m}(x,t)\nabla v_0(x,t)) = 0  &\quad \text{in }Q,\\
        \tilde{v}(x,T)=G^{(1)}_1(x)m^{(1)}_1(x,T)-G^{(1)}_2(x)m^{(1)}_2(x,T)&\quad \text{in }\Omega,\\
        \tilde{m}(x,0)=0 &\quad \text{in }\Omega,\\
        \tilde{v}=\nabla \tilde{v} = \tilde{m} = \nabla \tilde{m} = 0 &\quad \text{on }\Sigma,
    \end{cases}
    \end{equation} 
    since $F^{(1)}_1=F^{(1)}_2$. 
    Then, expanding $\Delta(\sigma \tilde{m})$ in the form $\sigma \Delta( m_i)  -2\nabla \sigma \cdot \nabla m_i-m_i\Delta \sigma$ and similarly for $\nabla \cdot (\tilde{m}\nabla v_0)$, by the unique continuation principle for general MFG systems as given in Theorem \ref{UCP}, 
    \[\tilde{v}=\tilde{m}\equiv0.\] In particular, \[v_1^{(1)}(x,T)=v_2^{(1)}(x,T) \quad \text{ and } \quad m_1^{(1)}(x,T)=m_2^{(1)}(x,T) \quad \text{in }\Omega\] for all solution pairs $(v_i,m_i)$. Picking a non-zero solution for $m_i$, $i=1,2$, we have that \[G^{(1)}_1(x)=G^{(1)}_2(x).\]

    We proceed to consider the second order linearization.
    By the previous step, we have that $(v_1^{(1)},m_1^{(1)})=(v_2^{(1)},m_2^{(1)})$, $(v_1^{(2)},m_1^{(2)})=(v_2^{(2)},m_2^{(2)})$ and $G^{(1)}_1(x)=G^{(1)}_2(x)$. Then, writing $\bar{v}=v^{(1,2)}_1-v^{(1,2)}_2$ and $\bar{m}=m^{(1,2)}_1-m^{(1,2)}_2$, we have
    \begin{equation}\label{eq:MFG2Diff2}
    \begin{cases}
        -\partial_t \bar{v}(x,t) -\sigma\Delta \bar{v}(x,t) + \kappa\nabla v_0(x,t) \cdot \nabla \bar{v}(x,t) 
        \\\qquad\qquad\qquad\qquad = F^{(1)}(x)\bar{m}(x,t) + [F^{(2)}_1(x,t) - F^{(2)}_2(x,t)]m^{(1)}(x,t)m^{(2)}(x,t) & \quad \text{in }Q,\\
        \partial_t \bar{m}(x,t) -\Delta(\sigma \bar{m}(x,t)) - c\nabla\cdot(\kappa \nabla \bar{v}(x,t)) - \nabla\cdot(\kappa\bar{m}(x,t) \nabla v_0(x,t)) = 0  &\quad \text{in }Q,\\
        \bar{v}(x,T)=G^{(1)}(x)\bar{m}(x,T)+[G^{(2)}_1(x)-G^{(2)}_2(x)]m^{(1)}(x,T)m^{(2)}(x,T)&\quad \text{in }\Omega,\\
        \bar{m}(x,0)=0 &\quad \text{in }\Omega,\\
        \bar{v}=\nabla \bar{v} = \bar{m} = \nabla \bar{m} = 0 &\quad \text{on }\Sigma.
    \end{cases}
    \end{equation}
    Since $m^{(1)},m^{(2)}$ have been uniquely obtained, we can view $E_i:=F^{(2)}_i(x,t)m^{(1)}(x,t)m^{(2)}(x,t)$ in \eqref{eq:UCPeq}, and once again apply the unique continuation principle Theorem \ref{UCP} to obtain $\bar{v}=\bar{m}\equiv0$. Substituting this into the first equation of \eqref{eq:MFG2Diff2}, we have that $E_1=E_2$. Choosing $m^{(1)},m^{(2)}\not\equiv0$, we obtain the result
    \[F^{(2)}_1(x,t)=F^{(2)}_2(x,t).\] In addition, 
    \[[G^{(2)}_1(x)-G^{(2)}_2(x)]m^{(1)}(x,T)m^{(2)}(x,T)=0,\] so choosing the same non-zero $m^{(1)},m^{(2)}$, we have that \[G^{(2)}_1(x)=G^{(2)}_2(x).\]

    Finally, by mathematical induction and repeating similar arguments as those in the first and second order linearization, one can show that $F^{(k)}_1(x,t)=F^{(k)}_2(x,t)$ and $G^{(k)}_1(x)=G^{(k)}_2(x)$ for all $k\in\mathbb{N}$. Hence we have the unique identifiability for the source functions and final values of $v$, i.e. $F_1=F_2$ and $G_1=G_2$.
    
    The proof is complete.
\end{proof}

\noindent\textbf{Acknowledgment.} 
	The work was supported by the Hong Kong RGC General Research Funds (No. 11311122, 11304224 and 11300821), the NSFC/RGC Joint Research Fund (No. N\_CityU101/21), and the ANR/RGC Joint Research Grant (No. A\_CityU203/19). We would also like to thank the handling editors and referees for their insightful comments.

\bibliographystyle{plain}
\bibliography{ref}
\end{document}